\documentclass[12pt,a4paper]{amsart}
\usepackage{amsmath,amsthm,amssymb,latexsym,a4wide,epic,eepic,bbm,mathrsfs,amsfonts,cases}
\usepackage{color,yfonts,indentfirst,bbold,dsfont}

\newtheorem{theorem}{Theorem}
\newtheorem{lemma}[theorem]{Lemma}
\newtheorem{corollary}[theorem]{Corollary}
\newtheorem{proposition}[theorem]{Proposition}
\newtheorem{example}[theorem]{Example}
\newtheorem{question}[theorem]{Question}

\usepackage[all]{xy}
\usepackage[active]{srcltx}
\usepackage[parfill]{parskip}
\usepackage{enumerate}
\numberwithin{equation}{section}

\newcommand{\tto}{\twoheadrightarrow}

\begin{document}

\title[Duflo involutions for tree quivers]{Duflo involutions for $2$-categories\\ associated to tree quivers}
\author{Xiaoting Zhang}

\begin{abstract}
Motivated by the definition of Duflo involution for fiat $2$-categories, we define certain analogues of
Duflo involution for arbitrary finitary $2$-categories and show that such Duflo involutions exist
for two classes of finitary $2$-categories associated with tree path algebras. Additionally, we
describe the quiver for the algebra underlying the principal $2$-representation for these
two classes of finitary $2$-categories.
\end{abstract}

\maketitle

\section{Introduction and description of results}\label{s1}

$2$-representation theory has its origins in the papers \cite{Kh,BFK,CR,KL,Ma,Ro} and is nowadays understood as
the study of $2$-representations of various kinds of $2$-categories. $2$-categorical analogues of
finite dimensional algebras are so-called {\em finitary} $2$-categories introduced in \cite{MM1}.
Basics of $2$-representation theory for finitary $2$-categories were developed in \cite{MM1,MM2,MM3,MM4,MM5,MM6}.

An important class of finitary $2$-categories is that of {\em fiat} $2$-categories, that is
finitary $2$-categories with weak involution and adjunction $2$-morphisms, see \cite{MM1}.
$2$-representation theory of fiat $2$-categories is much better understood than that in the
general case (compare the results from \cite{MM1,MM2,MM3,MM5,MM6} on fiat $2$-categories with
the results in \cite{MM2,MM4} on general finitary $2$-categories).
A major role in this theory is played by the so-called {\em Duflo involutions},
that is special indecomposable $1$-morphisms introduced in \cite{MM1}, which are important for
the definition and study of cell $2$-representations, that is ``simple'' $2$-representations of
fiat $2$-categories.

The combinatorics of a finitary $2$-category is determined by an algebraic structure called a
{\em multisemigroup}, whose ``structural'' units are called {\em cells} (left, right or two-sided).
In \cite{MM1} it is shown that each left cell of a fiat category contains a unique Duflo involution.
It is therefore natural to ask whether Duflo involutions can be defined  for general finitary $2$-categories.

In this paper, motivated by properties of Duflo involutions in the fiat case obtained in \cite{MM1}, we propose
an abstract definition of Duflo involution for arbitrary finitary $2$-categories. We are not able to
prove existence of Duflo involutions in the general case, however, we consider two natural classes of
finitary (not fiat) $2$-categories associated with tree path algebras (one of these classes originates in
\cite{GM1,GM} and the other one is new) and show that in these two cases Duflo involutions
do exist. However, a new feature is that these Duflo involution are no longer elements inside the
cell (in the general case) but can also exist outside the cell. We also check existence of
Duflo involution for the basic example of the $2$-category of projective functors associated to
any finite dimensional algebra.

Our approach is based on determination of combinatorial structure of all involved $2$-categories
together with explicit computation of projective presentation for simple objects in abelianized principal
$2$-representations of our finitary $2$-categories. As a bonus, we describe the quiver for the underlying
algebra of the principal $2$-representations for all these $2$-categories.

The paper is organized as follows: In Section~\ref{s2} we collect basic notions on $2$-categories and path algebras
for tree quivers. We define {\em signed Hasse diagram} to describe the poset of ideals for path algebras and
prove some auxiliary results. We also give an abstract definition of Duflo involutions for finitary $2$-categories.
In Section~\ref{s3}, we study the case of a $2$-category of dual projection functors associated to a tree quiver,
prove existence of Duflo involutions for all cells of this $2$-category and describe the quiver for the
path algebra underlying its principal $2$-representation. In Section~\ref{s4}, we study the
$2$-category $\mathscr{C}_B$ of projective bimodules over any basic, connected, finite dimensional
$\Bbbk$-algebra $B$ and determine Duflo involutions for all cells of $\mathscr{C}_B$.
We also describe a major part of the quiver for the path algebra underlying its principal $2$-representation.
In Section~\ref{s5} we define a new $2$-category, containing the $2$-category studied in
Section~\ref{s3} as a $2$-subcategory, and prove that this new $2$-category is a finitary $2$-category.
We describe all cells in this $2$-category and also prove existence of Duflo involutions for all left cells.
To describe the corresponding quiver for the path algebra underlying the principal $2$-representation, we
consider the ``natural'' $A$-$A$-bimodules obtained by tensoring the left and right ``natural''
$A$-modules. This bimodule happens to contain, as subbimodules,  all bimodules corresponding to
$1$-morphisms in our new $2$-category and in this way it provides crucial information about $2$-morphisms
in a very easy and unified way.
We also extend to this bimodules the partial order from Subsection~\ref{s2.2} and the corresponding
covering relations. Finally, we complete the paper with several  examples illustrating our results in Section~\ref{s6}.
\vspace{5mm}

\textbf{Acknowledgment.}
The paper was written during the visit of the author to Uppsala University supported by China Scholarship Council.
We thank Uppsala University for hospitality. We are very grateful to Volodymyr Mazorchuk
for useful discussions and valuable comments on the draft of this paper.

\section{Preliminaries}\label{s2}

\subsection{ Notation and setup}\label{s2.1}

We work over a fixed algebraically closed field $\Bbbk$. All categories and functors considered in this
note are assumed to be  $\Bbbk$-linear, that is enriched over $\Bbbk$-Mod. If not stated otherwise, by
a module we always mean a left module.

For a finite dimensional associative $\Bbbk$-algebra $A$ we denote by $A$-mod the (abelian) category of
all finitely generated $A$-modules and by mod-$A$ the (abelian) category of all finitely generated right $A$-modules.
Denote by $A$-mod-$A$ the category of all finitely generated $A$-$A$-bimodules. By an ideal we mean a two-sided ideal.

For an $2$-category $\mathscr{C}$, we keep the notational conventions from \cite{MM2}. We will denote objects of
$\mathscr{C}$ by $\mathtt{i},\mathtt{j}$ and so on; $1$-morphisms  of $\mathscr{C}$ by $F,G$ and so on,
$2$-morphisms of $\mathscr{C}$ by $\alpha,\beta$ and so on. The identity $1$-morphism in
$\mathscr{C}(\mathtt{i},\mathtt{i})$ will be denoted by $\mathbb{1}_\mathtt{i}$ for all objects $\mathtt{i}$
and composition of $1$-morphisms will be denoted by $\circ$.  Horizontal composition of $2$-morphisms will be denoted by $\circ_0$ and vertical
composition of $2$-morphisms will be denoted by $\circ_1$. For simplicity, we take $F(\alpha)$ and $\alpha_F$ short for $\mathrm{id}_F\circ_0\alpha$ and $\alpha\circ_0\mathrm{id}_F$ respectively.

\subsection{Path algebra of a tree quiver}\label{s2.2}

Let $Q=(Q_0, Q_1, s, t)$ be a finite {\em tree quiver}, that is a quiver whose underlying graph is a tree.
Here $Q_0$ is the set of vertices, $Q_1$ is the set of arrows, $s:Q_1\to Q_0$ is the source function and
$t:Q_1\to Q_0$ is the target function. Denote by $A=\Bbbk Q$ the corresponding path algebra. The algebra
$A$ is naturally graded by path lengths,
\begin{displaymath}
A=\bigoplus_{i\geq0}A_i,
\end{displaymath}
where $A_i$ is the linear span of all paths of length $i$ under the convention that each arrow in $Q_1$ has length one.
We denote by $Q_i$ the set of all paths in $Q$ of length $i$ and set
\begin{displaymath}
Q^p:=\bigcup_{i\geq 0}Q_i.
\end{displaymath}
The set $Q^p$ is finite. We denote by $\mathfrak{l}:Q^p\to \{0,1,2,\dots\}$ the {\em length function}, that is
the function which assigns the length of the path to each path.
For a vertex $v$ we denote by $\varepsilon_v$ the corresponding trivial path in $Q$ of length zero and in this
way we identify vertices in $Q$ with paths of length zero.

\begin{lemma}\label{lm1}
For each pair of vertices $i,j\in Q_0$, there exists at most one path between them.
\end{lemma}

\begin{proof}
This directly follows from the fact that the underlying graph for $Q$ is a tree.
\end{proof}

For $w,w'\in Q^p$ we write $w\preceq w'$ if $w'=wa$ or $aw$ for some $a\in Q^p$.
We also write $w\prec w'$ if $w\preceq w'$ and $w\neq w'$. Then
$\preceq$ is a partial order on $Q^p$. If $w\prec w'$ and there is no other path $u$ such that
$w\prec u\prec w'$, the we say that $w'$ \textit{covers} $w$. We denote by $C(w)$ the set of all paths covering $w$.
Note that the relation $w'$ covers $w$ implies that $\mathfrak{l}(w')=\mathfrak{l}(w)+1$.

Using the covering relation we may draw the {\em signed Hasse diagram} of the poset $Q^p$ in the following way:
\begin{itemize}
\item we write all paths of length zero in the top row, all paths of length one in row two, all paths of
length two in row three and so on;
\item we connect $w$ and $w'$ by a solid edge if $w'=aw$ for some $a\in Q_1$;
\item we connect $w$ and $w'$ by a dashed edge if $w'=wa$ for some $a\in Q_1$.
\end{itemize}

Note that our signed Hasse diagram contains slightly more information than the ordinary Hasse diagram of $Q^p$ (which
only encodes the covering relation).

\begin{example}\label{ex2}
{\rm
Let $A=\Bbbk Q$, where $Q$ is given by the left hand side of the following picture.
The right hand side then represents the signed Hasse diagram of $Q^p$.
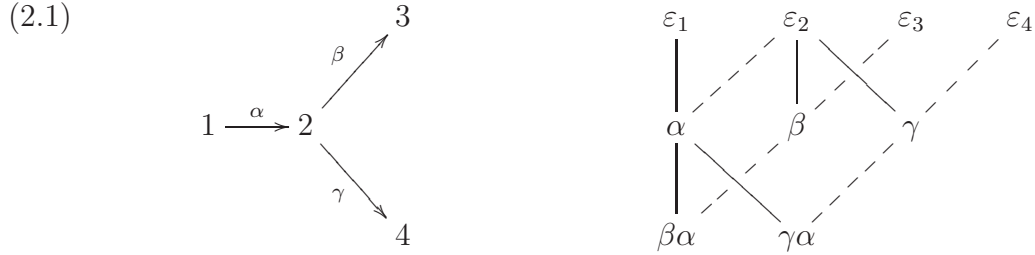
\begin{equation}\label{eq2.1}
\xymatrix{&&3&&&\varepsilon_1\ar@{-}[d] & \varepsilon_2 \ar@{-}[d]\ar@{--}[dl]\ar@{-}[dr] & \varepsilon_3\ar@{--}[dl]|\hole & \varepsilon_4\ar@{--}[dl]\\
1 \ar[r]^\alpha & 2\ar[ur]^\beta\ar[dr]_\gamma&&&&\alpha \ar@{-}[d]\ar@{-}[dr] & \beta \ar@{--}[dl]|\hole & \gamma  \ar@{--}[dl] &\\
&&4&&&\beta\alpha & \gamma\alpha  &&}
\end{equation}
From the signed Hasse diagram it is clear that the paths $\beta\alpha$ and $\gamma\alpha$ are the two maximal
elements in $Q^p$ and the paths $\varepsilon_1,\varepsilon_2,\varepsilon_3,\varepsilon_4$ are the
four minimal elements in $Q^p$.
}
\end{example}

For any subset $X\subset A$, we have the ideal $I_X=\langle X\rangle$
in $A$ defined as the minimal ideal of $A$ containing $X$ (or, the ideal of $A$ generated by $X$).

\begin{lemma}\label{lm3}
{\tiny\hspace{2mm}}
\begin{enumerate}[$($i$)$]
\item\label{lm3.1} The map $X\mapsto I_X$ is a bijection from the set of anti-chains in $Q^p$ to the set of
all two-sided ideals in $A$.
\item\label{lm3.2} For each anti-chain $X$ in $Q^p$, the set $X$ is a minimal set of generators in $I_X$ in the sense
that $I_Y\neq I_X$ for any proper subset $Y$ of $X$.
\end{enumerate}
\end{lemma}

\begin{proof}
Let $X$ and $Y$ be two different anti-chains in $Q^p$. Without loss of generality we may assume that
there is $x\in X$ such that $x\not\in Y$. If $x\not\in I_Y$, then $I_X\neq I_Y$ and we are done.
If $x\in I_Y$, then there is $y\in Y$ and a path $a$ such that $x=ay$ or $x=ya$. In each of this latter cases
we have $y\not\in I_X$ and hence again $I_X\neq I_Y$. This proves claim~\eqref{lm3.2}.
Claim~\eqref{lm3.2} implies injectivity of the map $X\mapsto I_X$ in claim~\eqref{lm3.1}.

Let $I$ be a non-zero ideal of $A$ and $x\in I$. Then, by Lemma~\ref{lm1}, for every $i,j\in Q_0$ we have that
$\varepsilon_i x\varepsilon_j$ a scalar multiple of an element in $Q^p$. If this scalar is non-zero, then the
corresponding element in $Q^p$ belongs to $I$. As the identity in $A$ is the sum of
all $\varepsilon_i$, it follows that each element in $I$ is a linear combination of paths in $I$.
Therefore $I=I_X$, where $X= I\cap Q^p$.

Now, for any $Y\subset Q^p$ and any $w,w'\in Y$ such that $w\prec w'$,
the ideal  $I_{Y\setminus\{w'\}}$ contains $w'$ (because it contains $w$) and hence $I_Y=I_{Y\setminus\{w'\}}$.
Since $Q^p$ is finite, $Y$ is finite as well, and thus $I_Y=I_Z$, where $Z$ is the set of all minimal elements in $Y$.
This proves that the map $X\mapsto I_X$ in claim~\eqref{lm3.1} is surjective and completes the proof.
\end{proof}

In what follows we will always assume that an ideal in $A$ is generated by some paths.
By Lemma~\ref{lm3}, we may assume that, moreover, there is a unique minimal way for such description of an ideal.

For an ideal $I$, denoted its minimal set of path generators by $G(I)$. The set $\mathcal{I}(A)$ consisting of all ideals of $A$ is partially ordered with respect to inclusions.
We denote by $\mathcal{I}(A)^{\mathrm{ind}}$ the subset of
$\mathcal{I}(A)$ consisting of all indecomposable ideals. Then $\mathcal{I}(A)^{\mathrm{ind}}$ inherits from
$\mathcal{I}(A)$ the poset structure. The covering relation for $\mathcal{I}(A)$ has the following
property:

\begin{lemma}\label{lm4}
Let $I\subset J$ be two ideals in $A$. Then $J$ covers $I$ if and only if $\dim(J/I)=1$.
\end{lemma}

\begin{proof}
The ``if'' part is clear, so we prove the ``only if'' part. Assume that $J$ covers $I$.
Let $X=I\cap Q^p$ and $Y= J\cap Q^p$. Then
$X\subset Y$ and $I\neq J$ implies $X\neq Y$. Let $y$ be a maximal element in $Y\setminus X$. Then the
ideal $I'$ generated by $I$ and $y$ properly contains $I$ and is contained in $J$. Hence $I'$ coincides with $J$
since $J$ covers $I$. This shows that $Y\setminus X=\{y\}$. Moreover, we also have that each element of
$I'$ is a linear combination of $y$ and elements in $X$. Therefore $\dim(I'/I)=1$, which completes the proof.
\end{proof}

Note that the covering relation in $\mathcal{I}(A)^{\mathrm{ind}}$ is different from that in $\mathcal{I}(A)$.
To distinguish them, we will call the covering relation in $\mathcal{I}(A)^{\mathrm{ind}}$ the {\em ind-covering}. For $I,J\in \mathcal{I}(A)^{\mathrm{ind}}$, by $I\Subset J$ we mean that $J$ ind-covers $I$.

If $I$ is an ideal and $G(I)=\{u_1,u_2,\ldots,u_k\}$, then for every $i=1,2,\dots,k$ the ideal $I$ covers
the ideal $I_X$, where
\begin{displaymath}
X=\{u_1,u_2,\dots,u_{i-1},u_{i+1},u_{i+2},\dots,u_k\}\cup C(u_i).
\end{displaymath}
Note that this $X$ is not necessarily a minimal set of generators for $I_X$. Indeed,
the ideal $I=\langle \varepsilon_1, \varepsilon_2\rangle$ in Example~\ref{ex2} covers
$\langle \varepsilon_1, \beta,\gamma\rangle$ and $\langle\varepsilon_2\rangle$ and we may observe that
$\{\varepsilon_1\}\cup C(\varepsilon_2)$ is not a minimal generating set for
$\langle \varepsilon_1, \beta,\gamma\rangle$ since $C(\varepsilon_2)=\{\alpha,\beta,\gamma\}$.

For each ideal $I$, set $Pa(I):=I\cap Q^p$. The set $Pa(I)$ is an upper set of the poset $Q^p$ and thus
inherits from $Q^p$ the structure of a poset. The signed Hasse diagram of $Pa(I)$ is a full subdiagram of the
signed Hasse diagram for $Q^p$.

\begin{lemma}\label{lm5}
An ideal $I$ is indecomposable if and only if its  signed Hasse diagram is connected.
\end{lemma}

\begin{proof}
This follows from the fact that, if $I=I_1\oplus I_2$, then the  signed Hasse diagram for $I$ is a disjoint union
of the signed Hasse diagrams for $I_1$ and $I_2$, and vise versa.
\end{proof}

\subsection{Finitary $2$-category}\label{s2.3}
Denote by \textbf{Cat} the category of all small categories.
By a $2$-category we mean a category enriched over \textbf{Cat}. Recall from \cite{MM1} that a $2$-category $\mathscr{C}$ is called \textit{finitary} over $\Bbbk$ provided that
\begin{itemize}
\item $\mathscr{C}$ has finitely many objects;
\item every $\mathscr{C}(\mathtt{i},\mathtt{j})$ is an idempotent split $\Bbbk$-linear category with finitely many isomorphism classes of indecomposable objects and finite dimensional spaces of morphisms;
\item all compositions are biadditive and $\Bbbk$-linear;
\item all identity $1$-morphisms are indecomposable.
\end{itemize}

\subsection{The multisemigroup associated to a finitary 2-category and cells}\label{s2.4}

For a finitary $2$-category $\mathscr{C}$, we denote by $\mathcal{S}_\mathscr{C}$ the set of isomorphism classes of all indecomposable $1$-morphisms in $\mathscr{C}$ with an added external zero element 0.
Following \cite{MM2}, the finite set $\mathcal{S}_\mathscr{C}$  becomes a \textit{multisemigroup} with the multiplication $\star $ defined for $[F],[G]\in\mathcal{S}_\mathscr{C}$ in the following way:
\begin{displaymath}
[F]\star[G]=\left\{\begin{array}{ll} \{[H]\in \mathcal{S}_\mathscr{C}|\ H\
\text{is isomorphic to a direct summand of}\ F\circ G \},  & F\circ G\neq 0; \\ 0, & \text{otherwise}.
\end{array}\right.
\end{displaymath}
This multisemigroup can be equipped with several natural preorders. We refer the reader to
\cite{KM} for more information on multisemigroups.

For two $1$-morphisms $F$ and $G$, we say $G\geq_L F$ in the \textit{left preorder} if there exists a $1$-morphism $H$ such
that $[G]\in [H]\star[F]$. We set $G\sim_L F$ if and only if $G\geq_L F$ and $G\leq_L F$, then
$\sim_L$ is an equivalent relation. A {\em left cell} is an equivalence class of $\sim_L$.
Similarly one defines the {\em right} and {\em two-sided} preorders $\geq_R$ and $\geq_J$ and the corresponding
{\em right} and {\em two-sided} cell by multiplying with $[H]$ on the right, respectively, with $[H]$ and $[H']$ on both sides.

\subsection{$2$-representations of $\mathscr{C}$}\label{s2.5}

Let $\mathscr{C}$ be a finitary $2$-category.
A $2$-representation of $\mathscr{C}$ is a $2$-functor to some other fixed $2$-category
(see \cite{Mc} for basics on $2$-categories and $2$-functors).
Important classes of $2$-representations are (see \cite{MM1,MM2} for details)

\begin{itemize}
\item {\em Finitary additive $2$-representations}, that is $2$-representations in which each object is represented by
an idempotent split $\Bbbk$-linear additive category with finitely many indecomposable objects, each
$1$-morphism is represented by an additive functor and each $2$-morphism is represented by a natural
transformation of functors.
\item {\em Abelian $2$-representations},  that is $2$-representations in which each object is represented by
a category equivalent to a module category for some finite dimensional associative $\Bbbk$-algebra, each
$1$-morphism is represented by an additive functor and each $2$-morphism is represented by a natural
transformation of functors.
\end{itemize}
$2$-representations of $\mathscr{C}$ form a $2$-category where the $1$-morphisms are $2$-natural transformations
and the $2$-morphisms are modifications, see \cite{MM1,MM2} for details. The category of
finitary additive $2$-representations is denoted by $\mathscr{C}$-add and the category of
finitary abelian $2$-representations is denoted by $\mathscr{C}$-mod.

Define the {\em abelianization} functor to be a $2$-functor from $\mathscr{C}\text{-add}$ to $\mathscr{C}\text{-mod}$, denoted by
\begin{displaymath}
\overline{\,\cdot\,}:\mathscr{C}\text{-add}\to \mathscr{C}\text{-mod},
\end{displaymath}
in the sense as defined in \cite[Subsection~4.2]{MM2}: for $\textbf{M}\in\mathscr{C}\text{-add}$ and $\mathtt{i}\in\mathscr{C}$, $\overline{\textbf{M}}(\mathtt{i})$ is a category with objects of all diagrams of the form $X\overset{\alpha}{\to}Y$, where $X,Y\in\textbf{M}(\mathtt{i})$ and $\alpha\in \mathrm{Hom}_{\textbf{M}(\mathtt{i})}(X,Y)$, and morphisms between two objects are commutative squares modulo factorization of the right downwards arrow, which makes the right downwards triangle commutes, using a homotopy. The $2$-action of $\mathscr{C}$ on $\overline{\textbf{M}}(\mathtt{i})$
is defined component-wise, and $2$-natural transformations and modifications in $\mathscr{C}\text{-add}$ can also be extended component-wise.

Let $\textbf{M}$, $\textbf{N}$ be two $2$-representations of $\mathscr{C}$, we say they are {\em equivalent} if there is a $2$-natural transformation between $\textbf{M}$ and $\textbf{N}$ such that the restriction of it to every object of $\mathscr{C}$ is an equivalence of categories. Recall from \cite[Proposition~2]{MM3}, this definition is compatible with the non-strict version of $2$-natural transformation, on which we won't focus in this paper.

To simplify the notation, we identify an indecomposable $1$-morphisms with its isomorphism class in $\mathcal{S}_\mathscr{C}$, writing $F\in\mathcal{S}_\mathscr{C}$ instead of $[F]\in\mathcal{S}_\mathscr{C}$.

Let $\mathcal{L}$ be a left cell. Since multiplication from the left does not change the source of the
original morphism, there is an $\mathtt{i}=\mathtt{i}_\mathcal{L}\in\mathscr{C}$ such that for any $1$-morphism $F\in \mathcal{L}$ we have $F\in\mathscr{C}(\mathtt{i},\mathtt{j})$ for some $\mathtt{j}\in\mathscr{C}$.

Consider the {\em principal} $2$-representation
\begin{displaymath}
\mathbf{P}_{\mathtt{i}}:= \mathscr{C}(\mathtt{i},{}_-):\mathscr{C}\to \mathbf{Cat},
\end{displaymath}
where for $\mathtt{j}\in\mathscr{C}$, the $2$-action of $\mathscr{C}$ on $\textbf{P}_{\mathtt{i}}(\mathtt{j})$
is given by the corresponding left horizontal composition. For any $\textbf{M}\in\mathscr{C}\text{-add}$ we have the Yoneda equivalence of categories, see \cite[Subsection~2.1]{Le},
\begin{displaymath}
\mathrm{Hom}_{\mathscr{C}\text{-add}}(\mathbf{P}_{\mathtt{i}},\textbf{M})\cong\textbf{M}(\mathtt{i}).
\end{displaymath}

Let $\overline{\mathbf{P}}_{\mathtt{i}}$ be the corresponding abelianization representation of $\textbf{P}_{\mathtt{i}}$ .
For an indecomposable 1-morphism $F\in\mathscr{C}(\mathtt{i},\mathtt{j})$ denote by $P_F$ the indecomposable projective module $0\longrightarrow F$ in $\overline{\mathbf{P}}_\mathtt{i}(\mathtt{j})$ and denote by $L_F$ its unique simple top.

For an additive category $\mathcal{C}$ and a set $B$ of objects in $\mathcal{C}$, we  denoted by $\mathrm{add}(B)$
the {\em additive closure} of $B$, that is the full subcategory of $\mathcal{C}$ consisting of all objects which
are isomorphic to direct summands of finite direct sums of objects from $B$.

Let $\mathcal{L}$ be a left cell in $\mathscr{C}$ and $\mathtt{i}=\mathtt{i}_\mathcal{L}\in\mathscr{C}$.
For $\mathtt{j}\in\mathscr{C}$ denote by $\mathbf{N}(\mathtt{j})$ the additive closure in $\mathbf{P}_\mathtt{i}(\mathtt{j})$ of all $1$-morphisms $F\in\mathscr{C}(\mathtt{i},\mathtt{j})$ such that $F\geq_L \mathcal{L}$.
Then $\mathbf{N}$ is a $2$-subrepresentation of $\mathbf{P}_\mathtt{i}$.
By \cite[Lemma~3]{MM5}, there exists a unique maximal ideal $\mathbf{I}$ in $\mathbf{N}$ such that it does not contain $\mathrm{id}_F$ for any $F\in\mathcal{L}$.
One defines the quotient $2$-functor $\mathbf{D}_\mathcal{L}:=\mathbf{N}/\mathbf{I}$, called the ({\em additive}) {\em cell} $2$-representations of $\mathscr{C}$ associated to $\mathcal{L}$.

\subsection{Cell $2$-representations for fiat $2$-categories.}\label{s2.6}

Given $\textbf{M}\in\mathscr{C}\text{-mod}$, $\mathtt{i}\in \mathscr{C}$ and $X\in\textbf{M}(\mathtt{i})$, for $\mathtt{j}\in \mathscr{C}$ define $\textbf{M}_{X}(\mathtt{j})$ to be $\mathrm{add}(FX)$, where $F$ runs through the set of all $1$-morphisms in $\mathscr{C}(\mathtt{i},\mathtt{j})$. The $2$-actions of $\mathscr{C}$ on $\textbf{M}_{X}$ are restrictions of that on $\textbf{M}$. Due to the finitary of $\mathscr{C}$ we have $\textbf{M}_{X}\in\mathscr{C}\text{-add}$, see \cite{MM2}.

A $\Bbbk$-finitary $2$-category $\mathscr{C}$ is called {\em fiat} if it has a weak object preserving involutive
anti-equivalence $\ast$ and for any $1$-morphism $F\in\mathscr{C}(\mathtt{i},\mathtt{j})$ there exist $2$-morphisms $\alpha:F\circ F^\ast\to\mathbb{1}_\mathtt{j}$ and
$\beta:\mathbb{1}_\mathtt{i}\to F^\ast\circ F$ such that $\alpha_F\circ_1F(\beta)=\mathrm{id}_F$ and $F^\ast(\alpha)\circ_1\beta_{F^\ast}=\mathrm{id}_{F^\ast}$.

By \cite[Proposition~17]{MM1}, when $\mathscr{C}$ is a fiat $2$-category, there exists a unique $G_\mathcal{L}\in\mathcal{L}$ (called the {\em Duflo involution}) such that the indecomposable projective module $P_{\mathbb{1}_\mathtt{i}}$ has a unique submodule $K_\mathcal{L}$ such that $K_\mathcal{L}$ has a simple top $L_{G_{\mathcal{L}}}$ and $FL_{G_{\mathcal{L}}}\neq0$ for any $F\in\mathcal{L}$, while each $F\in\mathcal{L}$ annihilates every simple subquotient of $P_{\mathbb{1}_\mathtt{i}}/K_\mathcal{L}$.
Set $U:=G_{\mathcal{L}}L_{G_\mathcal{L}}$, then the additive $2$-representation $\textbf{C}_\mathcal{L}:=(\overline{\mathbf{P}}_{\mathtt{i}})_U$ is called the {\em additive cell} $2$-representation of $\mathscr{C}$ with respect to $\mathcal{L}$. The abelianization $\overline{\textbf{C}}_\mathcal{L}$ of $\textbf{C}_\mathcal{L}$ is called the {\em abelian cell} $2$-representation of $\mathscr{C}$ with respect to $\mathcal{L}$.

Note that the additive cell $2$-representation $\textbf{C}_\mathcal{L}$ defined via the Duflo involution are equivalent to $\mathbf{D}_\mathcal{L}$ defined in Subsection\ref{s2.5} when $\mathscr{C}$ is a fiat $2$-category, for more details, see \cite[Proposition~22]{MM2}.

\subsection{Abstract Duflo involution}\label{s2.7}

Inspired by the previous subsection, we propose the following abstract notion for a Duflo involution.
Let $\mathscr{C}$ be a finitary $2$-category and $\mathcal{L}$ be a left cell in $\mathscr{C}$.
Let $\mathtt{i}:=\mathtt{i}_{\mathcal{L}}$.
We will say that an indecomposable $1$-morphism $G$ in $\mathscr{C}$
(which is not necessarily in $\mathcal{L}$) is a {\em Duflo involution} for $\mathcal{L}$ provided that
there is a submodule $K\subset P_{\mathbb{1}_\mathtt{i}}$ in $\overline{\mathbf{P}}_{\mathtt{i}}(\mathtt{i})$ such that
\begin{itemize}
\item $F L_{H}=0$ for all $F\in \mathcal{L}$ and all simple subquotients $L_H$ of $P_{\mathbb{1}_\mathtt{i}}/K$;
\item $K$ has a simple top isomorphic to $L_{G}$;
\item $F L_{G}\neq 0$ for all $F\in \mathcal{L}$.
\end{itemize}

From the previous subsection we know that for each left cell of a fiat category there is a Duflo involution,
moreover, this Duflo involution belongs to this left cell.

In the general case of finitary $2$-categories we do not know whether for each left cell there is
some Duflo involution. In the present paper we propose three different examples of finitary (not fiat)
$2$-categories for which we show that all left cells have Duflo involutions. Moreover, there are cases
when this Duflo involution is {\em not} an element of the corresponding left cell.

\section{Subbimodules of the identity bimodules for tree path algebras}\label{s3}

\subsection{Finitary $2$-category $\mathscr{D}_A$ for a tree path algebra}\label{s3.1}

Let now $A$ be the path algebra of a finite tree quiver $Q$ as described in Subsection~\ref{s2.2}.
Without loss of generality we may assume $Q_0=\{1,2,\ldots,n\}$, where $n$ is a positive integer.
The algebra $A$ is naturally an $A$-$A$-bimodule. We identify subbimodules of $_AA_A$ and ideals of $A$,
in particular, we will say that an ideal $I$ is {\em indecomposable} provided that it is indecomposable
as an $A$-$A$-bimodule.

For an ideal $I$ of $A$, denoted by $\mathrm{Dp}_{I}$ the functor
\begin{displaymath}
I\otimes_A{}_-:A\text{-mod}\to A\text{-mod}.
\end{displaymath}

Let $\mathcal{C}$ be a small category equivalent to $A$-mod. Define the $2$-category $\mathscr{D}_A$ to have
\begin{itemize}
\item one object $\mathtt{i}$ (which we identify with $\mathcal{C}$);
\item as $1$-morphisms, the functors given, up to equivalence with $A$-mod, by
functors from the additive closure of  all $\mathrm{Dp}_{I}$;
\item as $2$-morphisms, all  natural transformations of functors.
\end{itemize}

By \cite[Proposition~13]{GM}, the category $\mathscr{D}_A$ is a finitary $2$-category.
The category $\mathscr{D}_A$ is not a fiat $2$-category unless $Q$ has one vertex.

\subsection{Cells in  $\mathscr{D}_A$}\label{s3.2}

\begin{lemma}\label{lm6}
For each indecomposable ideal $I$ of $A$, the set $\{[\mathrm{Dp}_{I}]\}$ is a left cell,
a right cell and thus a two-sided cell as well.
\end{lemma}

\begin{proof}
We only give the proof of the statement for left cells. For right cells the proof is similar
and, put together, they give the statement for two-sided cells.

As $A$ is a path algebra with no relations, it is hereditary,  and thus $\mathrm{Dp}_{I}\circ \mathrm{Dp}_{J}=\mathrm{Dp}_{{IJ}}$ for any ideals $I,J$ of $A$, see \cite{GM}. Let $I,J$ be any two nonisomorphic indecomposable ideals of $A$. Assume that $\mathrm{Dp}_{I}\sim_L\mathrm{Dp}_{J}$. By definition, there exist indecomposable ideals $H,H'$ such that $I$ is isomorphic to a direct summand of $HJ\subset J$ and $J$ is
isomorphic to a direct summand of $H'I\subset I$. This means that there are injective bimodule
homomorphisms $J\hookrightarrow I\hookrightarrow J$ and hence  $I\cong J$.
The claim follows.
\end{proof}

\subsection{Quiver for the underlying algebra for the principal $2$-representation of $\mathscr{D}_A$}\label{s3.3}

The aim of this subsection is to describe the quiver of $\mathscr{D}_A(\mathtt{i},\mathtt{i})$.
For ideals $J$ and $J'$ in $A$ such that $J\subset J'$ we denote by $\iota_{(J,J')}:J\to J'$ the canonical
inclusion. We start with the following observation.

\begin{lemma}\label{lm7}
Let $B$ be a finite dimensional algebra and $M$ be a $B$-module with all composition multiplicities $\leq 1$.
Then for any indecomposable submodule $N$ and any submodule $K$ in $M$ we have:
\begin{enumerate}[$($i$)$]
\item\label{lm7.1} Any non-zero homomorphism from $N$ to $K$ is injective.
\item\label{lm7.2}$\mathrm{Hom}_{B}(N,K)=
\left\{\begin{array}{ll} \Bbbk\iota_{(N,K)},   & if\ N\subset K;\\ 0, & if\ N\not\subset K,
\end{array}\right.$
where $\iota_{(N,K)}$ denotes the natural inclusion.
\end{enumerate}
\end{lemma}

\begin{proof}
Let $\varphi:N\to K$ be a non-zero homomorphism, $X:=\mathrm{Ker}(\varphi)$ and $Y:=\mathrm{Im}(\varphi)\neq 0$.
Since all composition multiplicities in $M$ are $\leq 1$, the submodule $Y$ of $M$ belongs to $N$.
This implies that is $N=X\oplus Y$. As $N$ is indecomposable and $Y\neq 0$, we obtain $X=0$
which proves claim~\eqref{lm7.1}.

If $N\not\subset K$, then the assumption that all composition multiplicities in $M$ are $\leq 1$ implies
that $N$ has a simple subquotient $L$ which is not a subquotient of $K$. Then $L$ must be annihilated by
any homomorphism from $N$ to $K$ and hence claim~\eqref{lm7.1} implies $\mathrm{Hom}_{B}(N,K)=0$.

Assume that $N\subset K$ and let $\varphi,\psi:N\to K$ be non-zero homomorphisms. Then they both are injective
by claim~\eqref{lm7.1}, in particular, they are injective, when restricted to the socle of $N$.
Let $L$ be a simple subquotient in the socle of $N$. The assumption that all composition multiplicities in $M$
are $\leq 1$ implies that $\varphi$ and $\psi$ both induce isomorphisms when restricted to $L$.
Let $a\varphi\vert_{L}+b\psi\vert_{L}$ be a non-trivial linear combination which annihilates $L$
(it exists by Schur's lemma). But then claim~\eqref{lm7.1} implies that $a\varphi+b\psi$, whose kernel contains
$L$, must be the zero homomorphism. Therefore the space $\mathrm{Hom}_{B}(N,K)$ is at most one-dimensional.
Claim~\eqref{lm7.2} follows.
\end{proof}

\begin{corollary}\label{cr8}
For $I, J\in \mathcal{I}(A)^\mathrm{ind}$,  we have:
\begin{enumerate}[$($i$)$]
\item\label{cr8.1} Any non-zero homomorphism from $I$ to $J$ is injective.
\item\label{cr8.2}$\mathrm{Hom}_{A\text{-}A}(I,J)=\left\{\begin{array}{ll} \Bbbk\iota_{(I,J)},   & if\ I\subset J;\\ 0, & if\ I\not\subset J.
\end{array}\right.$
\end{enumerate}
\end{corollary}

\begin{proof}
By Lemma~\ref{lm1}, the quiver $Q$ contains at most one oriented path between each pair of vertices.
This means that $\varepsilon_{i}A\varepsilon_{j}$ is at most $1$-dimensional for all $i,j\in Q_0$,
that is, all composition multiplicities of $A$, as $A$-$A$-bimodule, are at most $1$.
Therefore the assertion of the corollary follows from Lemma~\ref{lm7}.
\end{proof}

Using Corollary~\ref{cr8}, we can determine the quiver $\mathcal{Q}^{(1)}$ for the underlying algebra of the principal
$2$-representation $\mathbf{P}_{\texttt{i}}$ of $\mathscr{D}_A$. The vertices of $\mathcal{Q}^{(1)}$
are given  by indecomposable subbimodules in ${}_AA_A$. For two indecomposable subbimodules
$I,J$ in ${}_AA_A$ there is exactly one arrow from $I$ to $J$ if $J\Subset I$
and there are no arrows otherwise
(note that arrows in the quiver go in the opposite direction than homomorphisms between the corresponding
projective modules). Furthermore, for any chains of indecomposable ideals
\begin{displaymath}
I_1\Subset I_2 \Subset\dots \Subset I_k\quad\text{ and }\quad
J_1\Subset J_2 \Subset\dots \Subset J_m
\end{displaymath}
such that $I_1=J_1$ and $I_k=J_m$, we have
\begin{displaymath}
\iota_{(I_{k-1},I_k)}\cdots\iota_{(I_2,I_3)}\iota_{(I_1,I_2)} =
\iota_{(J_{m-1},J_m)}\cdots\iota_{(J_2,J_3)}\iota_{(J_1,J_2)}.
\end{displaymath}
Therefore we have to impose all commutativity relations for $\mathcal{Q}^{(1)}$ which make sense. Since all
$\iota_{(I,J)}$ are injective, no other relations are necessary.
Explicit examples are given in Subsection \ref{s6.2}.

One corollary from this description is that, thanks to the commutativity relations, all
indecomposable projective modules over the quiver algebra of $\mathcal{Q}^{(1)}$ with imposed
commutativity relations are multiplicity free.

\subsection{Duflo involution of $\{[\mathrm{Dp}_{I}]\}$}\label{s3.4}

In this subsection we verify that $[\mathrm{Dp}_{J}]$ defined below is the abstract Duflo involution corresponding to  the left cell $\{[\mathrm{Dp}_{I}]\}$.

\begin{proposition}\label{pr9}
For $J\in \mathcal{I}(A)^{\mathrm{ind}}$, the homomorphism
\begin{equation}\label{eq3.1}
\left(\bigoplus_{J\text{ covers } X}P_{\mathrm{Dp}_{X}}\right)\overset{\mathbf{d}}{\rightarrow}P_{\mathrm{Dp}_{J}}
\end{equation}
where $\mathbf{d}$ is a $1$-row matrix with coefficients $\iota_{(X,J)}$,
gives a projective presentation of $L_{\mathrm{Dp}_{J}}$ in  $\overline{\mathbf{P}}_{\mathtt{i}}(\mathtt{i})$.
\end{proposition}

\begin{proof}
Let $I,J$ be indecomposable ideals in $A$ and $\varphi:I\to \mathrm{Rad}(J)$ be a non-zero map. Then $\varphi$
is injective by Corollary~\ref{cr8}\eqref{cr8.1} and its image is contained in a maximal subideal
$X$ of $J$. By maximality, we have that $J$ covers $X$. From the previous subsection we have that
$P_{\mathrm{Dp}_{J}}$ is multiplicity free. It follows that the sum of all subobjects of the
form $P_{\mathrm{Dp}_{X}}$, where $J$ covers $X$, coincides with the radical of
$P_{\mathrm{Dp}_{J}}$. The claim follows.
\end{proof}

Let $I$ be an indecomposable ideal of $A$ and $G(I)=\{u_1,u_2,\ldots,u_k\}$.
Denote by $s_{G(I)}$ the set of sources of all generators in $G(I)$. Set $J:=\langle \varepsilon_{i} |\ i\in s_{G(I)}\rangle$.

\begin{lemma}\label{lm10}
The ideal  $J$ is indecomposable,  $I\subset J$ and $IJ=I$.
\end{lemma}

\begin{proof}
That  $I\subset J$ and $IJ=I$ follows directly from the definitions.

Note that $G(J)=\{\varepsilon_{i} |\ i\in s_{G(I)}\}$. Assume that $J=J'\oplus J''$, where $J'$, $J''$ are two proper
subideals in $J$. By Lemma~\ref{lm3} we have $G(J)=G(J')\cup G(J'')$, $G(J')=G(J)\cap J'$ and $G(J'')=G(J)\cap J''$.

Since $I=IJ=IJ'\oplus IJ''$ and $I$ is indecomposable, we get $IJ'=0$ (or $IJ''=0$). Recalling that $J''$ is proper, there exists $\varepsilon_j\in J'$ for some $ j\in s_{G(I)}$.
Thus those $u_i$'s, whose source is $j$, lie in $IJ'$ as $u_i\in I,\ \varepsilon_j\in J'$. This implies that $IJ'\neq0$ and one can also obtain $IJ''\neq0$ similarly, which is a contradiction.
\end{proof}

\begin{lemma}\label{lm11}
Given $ I,J $ as above, we have $\mathrm{Dp}_{I}L_{\mathrm{Dp}_{J}}\neq 0$.
\end{lemma}

\begin{proof}
Assume that $X$ is an ideal covered by $J$. Then there exists exactly one $\varepsilon_{j}$ with $j\in s_{G(I)}$,
which is not contained in $X$. Assume that
\begin{displaymath}
\{u_{i_1},u_{i_2},\dots,u_{i_m}\}, \qquad\text{where}\qquad i_1<i_2<\dots<i_m,
\end{displaymath}
is the list of all elements in $G(I)$ with source $j$. This list is not empty.

We claim that all $u_{i_t}$'s do not belong to $IX$. Indeed, if $u_{i_t}\in IX$, then $u_{i_t}=au_lbx$
where $a,b\in Q^p$ and $x\in Pa(X)$. Since different elements in $G(I)$ are not comparable,
we get $l=i_t$ and $\mathfrak{l}(a)=\mathfrak{l}(b)=\mathfrak{l}(x)=0$. This forces
$x=\varepsilon_{j}$, a contradiction. Thus we have $IX\subsetneq I$.

Multiplying \eqref{eq3.1} with $I$ from the left, we get
\begin{equation}\label{eq3.2}
\bigoplus_{J\text{ covers } X}P_{\mathrm{Dp}_{IX}}\overset{\mathbf{d}'}{\rightarrow}P_{\mathrm{Dp}_{I}},
\end{equation}
where $\mathbf{d}'$ is a $1$-row matrix with coefficients $\iota_{(IX,I)}$. Since each
$IX\neq I$, the object in $\overline{\mathbf{P}}_{\mathtt{i}}(\mathtt{i})$ corresponding to \eqref{eq3.2}
is non-zero. This implies the claim of the lemma.
\end{proof}

\begin{lemma}\label{lm12}
For any indecomposable ideal $J'\not\subset J$, we have $\mathrm{Dp}_{I}L_{\mathrm{Dp}_{{J'}}}=0$.
\end{lemma}

\begin{proof}
Let $G(J')=\{w_1,w_2,\ldots,w_l\}$. Then there exists $w_j\not\in J$ for some $j$. Moreover, since $J$ is generated by
elements of length zero, neither $s(w_j)$ nor $t(w_j)$ are in $s_{G(I)}$. Hence $u_iw_j=0$ for all $1\leq i\leq k$.

Consider the following projective  presentation of $L_{\mathrm{Dp}_{J'}}$ in
$\overline{\mathbf{P}}_{\mathtt{i}}(\mathtt{i})$:
\begin{displaymath}
\bigoplus_{J'\text{ covers } X'}P_{\mathrm{Dp}_{X'}}\overset{\mathbf{c}}{\rightarrow}P_{\mathrm{Dp}_{J'}}
\end{displaymath}
where $\mathbf{c}$ is a $1$-row matrix with coefficients $\iota_{(X',J')}$. Multiplying this from the left with $I$,
we get
\begin{equation}\label{eq3.3}
\bigoplus_{J'\text{ covers } X'}P_{\mathrm{Dp}_{IX'}}\overset{\mathbf{c}'}{\rightarrow}P_{\mathrm{Dp}_{IJ'}}
\end{equation}
where $\mathbf{c}'$ is a $1$-row matrix with coefficients $\iota_{(IX',IJ')}$.

The ideal $IJ'$ is generated by the set $\{u_iaw_t|\ 1\leq i\leq k, 1\leq t\leq l, a\in Q^p\}$. If there exists some $X'$
covered by $J'$ such that all these $u_iaw_t$'s lie in $IX'$, then $IX'\supset IJ'$. It is clear that $IX'\subset IJ'$,
which implies that $IX'=IJ'$. This means that the corresponding map $\iota_{(IX',IJ')}$ in \eqref{eq3.3} is the identity map and hence \eqref{eq3.3} represents the zero object in $\overline{\mathbf{P}}_{\mathtt{i}}(\mathtt{i})$.

It remains to show that such a $X'$ as desired in the previous paragraph exists. In fact, we show that we can take $X'$ as
the ideal covered by $J'$ which does not contain $w_j$. Then $w_j$ is the only path
in $J'$ which is not in $X'$ by Lemma~\ref{lm4}. Let us compare $ IX' $ with $ IJ' $.
Clearly, $u_iaw_t\in IX'$ for all $t\neq j$. If $u_iaw_j=0$ for $1\leq i\leq k$ and all $a\in Q^p$, then we are done.
If there exists some $b\in Q^p$ and $i\in\{1,2,\dots,k\}$ such that $u_ibw_j\neq0$, then $\mathfrak{l}(b)>0$ due to $u_iw_j=0$. Write $b$ as a composition of arrows, say, $b=b_lb_{l-1}\ldots b_2b_1$. Then $b_1w_j\in X'$ and
$u_ibw_j\in IX'$. Thus $IX'$ contains all $u_iaw_t$'s and the proof is complete.
\end{proof}

\begin{theorem}\label{thm13}
Given $I, J$ as above, the element $\mathrm{Dp}_{J}$ is the Duflo involution corresponding to $\{[\mathrm{Dp}_{I}]\}$.
\end{theorem}

\begin{proof}
By Lemma~\ref{lm11}  and Lemma~\ref{lm12}, we observe that $ P_{\mathrm{Dp}_{J}} $ is not annihilated by
$\mathrm{Dp}_{I} $ but every simple subquotient $P_{\mathbb{1}_\texttt{i}}/P_{\mathrm{Dp}_{J}}$
is annihilated by $ \mathrm{Dp}_{I} $. Thus with respect to the left cell
$\{[\mathrm{Dp}_{I}]\}$, the minimal submodule $K_\mathcal{L}$ of $P_{\mathbb{1}_\texttt{i}}$ in
$\overline{\mathbf{P}}_{\texttt{i}}$ is $P_{\mathrm{Dp}_{J}}$, which has a simple top
$L_{\mathrm{Dp}_{J}}$. This completes the proof of the theorem.
\end{proof}

\section{Duflo involutions for projective bimodules}\label{s4}

\subsection{The $2$-category of projective bimodules}\label{s4.1}

Let $B$ be any basic, connected, finite dimensional $\Bbbk$-algebra and $\mathcal{B}$ be a small category
equivalent to $B$-mod. Following \cite[Subsection~7.3]{MM1}, define the finitary $2$-category $\mathscr{C}_B$ to have
\begin{itemize}
\item one object $\mathtt{i}$ (identified with the category $\mathcal{B}$);
\item as $1$-morphisms, the functors given, up to equivalence with $B$-mod, by functors from the additive closure of ${}_BB\otimes_\Bbbk B_B\otimes_B{}_-$ and the identity functor $\mathrm{Id}_{\mathcal{B}}$;
\item as $2$-morphisms, all  natural transformations of functors.
\end{itemize}

Let $\{e_1, e_2,\dots, e_n\}$ be a complete and irredundant set of primitive and pairwise orthogonal idempotents in $B$.
For $i,j\in\{1,2,\dots,n\}$, denoted by $Y_{ij}$ the projective $B$-$B$-bimodule $Be_i\otimes_\Bbbk e_jB$
and by $G_{ij}$ the $1$-morphism in $\mathscr{C}_B$ corresponding to tensoring with $Y_{ij}$.
Let $L_{ij}$ be the simple top of $Y_{ij}$ (as a $B$-$B$-bimodule). By \cite[Subsection~7.3]{MM1}, we have
\begin{displaymath}
G_{ij}G_{st}=G_{it}^{\oplus \dim(e_jBe_s)}.
\end{displaymath}
For $1\leq j\leq n$, the set $\mathcal{L}_j:=\{[Y_{ij}]|\ 1\leq i\leq n\}$ is a left cell. For $1\leq i\leq n$, the set $\mathcal{R}_i:=\{[Y_{ij}]|\ 1\leq j\leq n\}$ is a right cell.
The set $\mathcal{J}:=\{[Y_{ij}]|\ 1\leq i,j\leq n\}$ is a two-sided cell which is maximal with respect to $\geq_{\mathcal{J}}$. If $B$ is not simple, then, apart from $\mathcal{J}$, there is only one more
two-sided cell and it consists of only one element, namely, $H:=\mathrm{Id}_{\mathcal{B}}$.
If $B$ is simple, then $n=1$ and $G_{11}\cong \mathrm{Id}_{\mathcal{B}}$.

\subsection{Duflo involutions for projective bimodules}\label{s4.2}

\begin{theorem}\label{thm14}
For $1\leq j\leq n$, the element $G_{jj}$ is the Duflo involution in $\mathcal{L}_j$.
\end{theorem}

\begin{proof}
First we claim that $G_{jj} L_{H}=0$. For $s,t\in\{1,2,\dots,n\}$, let us choose a basis
$\varphi_{1}^{(s,t)},\varphi_{2}^{(s,t)},\dots,\varphi_{k_{st}}^{(s,t)}$ in
$\mathrm{Hom}_{B\text{-}B}(Y_{st},B)$. We also choose some basis $\psi_1,\psi_2,\dots,\psi_k$ in
$\mathrm{Hom}_{B\text{-}B}(B,\mathrm{Rad}(B))$. Then the map
\begin{equation}\label{eq4.1}
\bigoplus_{s,t=1}^n\bigoplus_{l=1}^{k_{st}} Y_{st}^{(l)}
\oplus \bigoplus_{l=1}^{k} B^{(l)}\overset{\Phi}{\longrightarrow} B,
\end{equation}
where $\Phi$ is the direct sum of all $\varphi_{l}^{(s,t)}$ and all $\psi_l$, corresponding to a
projective presentation of $L_{H}$. Note that the map $\Phi$ in \eqref{eq4.1} is surjective in $B$-mod-$B$, since $\bigoplus_{s,t} Y_{st}$ is a projective generator of the
latter category. Applying $Y_{jj}\otimes_{B}{}_-$ to \eqref{eq4.1}, we get a surjective map
\begin{equation}\label{eq4.2}
\bigoplus_{s,t=1}^n\bigoplus_{l=1}^{k_{st}} Y_{jj}\otimes_B Y_{st}^{(l)}
\oplus \bigoplus_{l=1}^{k} Y_{jj}^{(l)}\overset{Y_{jj}\otimes_B \Phi}{\longrightarrow} Y_{jj}.
\end{equation}
As $Y_{jj}$ is a projective $B$-$B$-bimodule, \eqref{eq4.2} splits and hence realizes the
zero object in the abelianization $\overline{\mathbf{P}}_{\mathtt{i}}$. Hence $G_{jj} L_{H}=0$.

Next we claim that $G_{jj} L_{G_{st}}=0$ for all $s,t\in \{1,2,\dots,n\}$ such that  $s\neq j$.
The projective presentation of $L_{G_{st}}$ is given by a projective presentation
\begin{displaymath}
X\to Y_{st}\tto L_{st}
\end{displaymath}
of $L_{st}$ in $B$-mod-$B$. Applying $Y_{jj}\otimes_{B}{}_-$ to this presentation,
and observing that $Y_{jj}\otimes_{B}L_{st}=0$ if $j\neq s$,
similarly to the previous paragraph it follows that $G_{jj} L_{G_{st}}=0$.

Last we claim that $G_{jj} L_{G_{jj}}\neq 0$.
The projective presentation of $L_{G_{jj}}$ is given by a projective presentation
\begin{displaymath}
X\to Y_{jj}\tto L_{jj}
\end{displaymath}
of $L_{jj}$ in $B$-mod-$B$. Applying $Y_{jj}\otimes_{B}{}_-$ to this presentation,
we get
\begin{displaymath}
Y_{jj}\otimes_{B} X\to Y_{jj}\otimes_{B} Y_{jj}\tto Y_{jj}\otimes_{B} L_{jj}.
\end{displaymath}
Since both $Y_{jj}\otimes_{B} X$ and $Y_{jj}\otimes_{B} Y_{jj}$ are projective bimodules,
the observation that we have $Y_{jj}\otimes_{B}L_{jj}\neq 0$ implies that $G_{jj} L_{G_{jj}}\neq 0$.

Consider now the subbimodule $Be_jB\subset B$. Then all composition subquotients of $B/Be_jB$
have the form $L_{st}$ for some $s,t\in \{1,2,\dots,n\}$ such that $s,t\neq j$.
We also have a surjective $B$-$B$-bimodule homomorphism from $Be_j\otimes_\Bbbk e_jB$ to $Be_jB$
sending $e_j\otimes e_j$ to $e_j$. This means that $Be_jB$ has simple top and this top is isomorphic
to $L_{jj}$. Comparing all these with the definition of a Duflo involution, we conclude that
$G_{jj}$ is indeed the Duflo involution in the cell $\mathcal{L}_j$.
\end{proof}

\subsection{The quiver of the principal $2$-representation}\label{s4.3}

It is easy to describe a major part of the quiver for the algebra underlying the principal
$2$-representation $\mathbf{P}_{\mathtt{i}}$ of $\mathscr{C}_B$. As $B$ is basic, it is isomorphic to
the quotient of the path algebra of some quiver $Q$ modulo an admissible ideal $I$.
Then the endomorphism algebra $\bigoplus_{s,t=1}^n Y_{st}$ is isomorphic to $B\otimes_{\Bbbk}B^{\mathrm{op}}$
and hence is given by the quiver $Q\times Q^{\mathrm{op}}$ modulo the ideal generated by
$I$, the opposite  of $I$ and by all possible commutativity relations between elements of $Q$ and
$Q^{\mathrm{op}}$.

Further, any homomorphism from a projective bimodule to the identity bimodule factors through the
surjection
\begin{displaymath}
\bigoplus_{j=1}^n Y_{jj}\to B,
\end{displaymath}
where $Y_{jj}\to B$ is given by sending $e_j\otimes e_j$ to $e_j$. This means that the quiver
for the algebra underlying the principal $2$-representation $\mathbf{P}_{\mathtt{i}}$ of $\mathscr{C}_B$
contains one arrow from the vertex corresponding to $L_{H}$ to the vertex corresponding to $L_{G_{jj}}$,
for each $j$ and no arrows to any of $L_{G_{st}}$ for $s\neq t$.

We do not know how to determine arrows from $L_{H}$ to itself or from $L_{G_{st}}$ to $L_H$ in the general
case. In some special case this is possible, see, for example, Subsection~\ref{s5.5}.

\section{A new finitary $2$-category for tree path algebras}\label{s5}

\subsection{Definition}\label{s5.1}

In this subsection, we work in the setup of Subsection~\ref{s3.1}.
Define the $2$-category $\mathscr{D}_A'$ to have
\begin{itemize}
\item one object $\mathtt{i}$ (identified with the category $\mathcal{C}$ in Subsection \ref{s3.1});
\item as $1$-morphisms, the functors given, up to equivalence with $A$-mod, by functors from the additive closure of $_AA\otimes_\Bbbk A_A\otimes_A{}_-$ and all $\mathrm{Dp}_{I}$;
\item as $2$-morphisms, all  natural transformations of functors.
\end{itemize}
The $2$-category $\mathscr{D}_A'$ is not fiat in general. However, it is finitary as shown in the following proposition.

\begin{proposition}\label{pr15}
The category $\mathscr{D}_A'$ is a finitary $2$-category.
\end{proposition}

\begin{proof}
By definition, $\mathscr{D}_A'$ only has one object. Due to the connectedness of the tree algebra $A$, the identity functor
$\mathbb{1}_\mathtt{i}=\mathrm{Dp}_{A}$ is indecomposable. Note that $A$ is finite dimensional, therefore
the functor $_AA\otimes_\Bbbk A_A\otimes_A{}_-$ decomposes into a finite direct sum of projective functors.
From \cite[Corollary~11]{GM}, $A$ has finitely many ideals. Moreover, each ideal in $A$
is projective both as a left and as a right $A$-module since $A$ is hereditary.
It is clear that both $I\otimes_AA\cong I$ and $ I\cong A \otimes_A I$ as $A$-$A$-bimodules. This implies that
direct summands of both
\begin{displaymath}
\mathrm{Dp}_{I}\circ(_AA\otimes_\Bbbk A_A\otimes_A{}_-)\qquad\text{ and }\qquad
(_AA\otimes_\Bbbk A_A\otimes_A{}_-)\circ\mathrm{Dp}_{I}
\end{displaymath}
are isomorphic to direct summands of $_AA\otimes_\Bbbk A_A\otimes_A{}_-$.
Therefore the category $\mathscr{D}_A'(\mathtt{i},\mathtt{i})$ has finitely many indecomposable $1$-morphisms
(up to isomorphism). Spaces of $2$-morphisms are just $A$-$A$-bimodule homomorphisms between corresponding finite dimensional $A$-$A$-bimodules, hence have finite dimension.
\end{proof}

\subsection{Cells in $\mathscr{D}_A'$}\label{s5.2}

For simplicity, we denote by $F_{ij}$ the projective functors $A\varepsilon_i\otimes_\Bbbk \varepsilon_jA\otimes_A{}_-$
and by $X_{ij}$ the corresponding $A$-$A$-bimodule $A\varepsilon_i\otimes_\Bbbk \varepsilon_jA$, where $i,j\in Q_0$.

\begin{lemma}\label{lm16}
The list
\begin{equation}\label{eq5.1}
\{[\mathrm{Dp}_{I}], [F_{ij}]\ |\ I\in \mathcal{I}(A)^\mathrm{ind}
\text{ and } |G(I)|\neq1; i,j\in Q_0\}
\end{equation}
is a complete and irredundant list of elements in $\mathcal{S}_{\mathscr{D}_A'}$.
\end{lemma}

\begin{proof}
The set $\{F_{ij}|\ 1\leq i,j\leq n\}$ is a complete and irredundant list of direct summands of the projective functor
$_AA\otimes_\Bbbk A_A\otimes_A{}_-$. Note that
\begin{equation}\label{eq5.2}
F_{ij}F_{st}=F_{it}^{\oplus \dim(\varepsilon_jA\varepsilon_s)}.
\end{equation}
When $|G(I)|=1 $, we have $I=\langle a\rangle$ for some $a\in Q^p$.
There is a unique $A$-$A$-bimodule homomorphism from the projective bimodule $X_{t(a)s(a)}$ to $I$
sending $\varepsilon_{t(a)}\otimes\varepsilon_{s(a)}$ to $a$. This homomorphism is surjective as
$I$ is generated by $a$. Comparing the dimensions of $I$ and $X_{t(a)s(a)}$, we see that this
homomorphism is, in fact, bijective. Hence $X_{t(a)s(a)}\cong I$.

When $|G(I)|\neq1$, then the cardinality of the minimal generating set for $I$ is strictly greater than $1$,
see Lemma~\ref{lm3}\eqref{lm3.2}. Therefore $\mathrm{Dp}_{I}$ is not isomorphic to any $F_{ij}$
since, for the latter, the corresponding projective $A$-$A$-bimodule $X_{ij}$ is generated by one element.
The claim follows.
\end{proof}

\begin{lemma}\label{lm17}
All cells in $ \mathcal{S}_{\mathscr{D}_A'} $ are listed as follows:
\begin{enumerate}[$($i$)$]
\item\label{lm17.1} For $I\in \mathcal{I}(A)^\mathrm{ind}$ such that $|G(I)|\neq1$, the set
$\{[\mathrm{Dp}_{I}]\}$ is a left cell, a right cell and thus a two-sided cell.
\item\label{lm17.2} For $j\in Q_0$, the set $\{[F_{ij}]|\ i\in Q_0\}$ is a left cell;
for $i\in Q_0$, the set  $\{[F_{ij}]|\ j\in Q_0\}$ is a right cell; and
the set $\{[F_{ij}]|\ i,j\in Q_0\}$ is a
two-sided cell which is maximal with respect to $\geq_J$.
\end{enumerate}
\end{lemma}

\begin{proof}
Note that $F_{ij}\circ\mathrm{Dp}_{I}\cong A\varepsilon_i\otimes_\Bbbk \varepsilon_jI\otimes_A-$, where
$\varepsilon_jI$ is projective as right $A$-module as $A$ is hereditary.
Assume that $G(I)=\{u_1,u_2,\ldots,u_k\}$, then $\varepsilon_jI$ is generated, as right $A$-module,
by the set $\{\varepsilon_jau_i|\ a\in A, 1\leq i\leq k\}=:Y$. Let $\{w_1,w_2,\ldots,w_l\}$ be the
set of all elements in $Y$ which are minimal with respect to $\preceq$. We have $t(w_i)=j$ for all $1\leq i\leq l$
and $\{w_1,w_2,\ldots,w_l\}$ is an anti-chain with respect to $\preceq$.
Then, similarly to the proof of Lemma~\ref{lm3}, we have that
$\varepsilon_jI\cong \oplus_{1\leq i\leq l} w_iA$ as right $A$-modules.

Each $w_iA$ is isomorphic to the projective right $A$-module $\varepsilon_{s(w_i)}A$.
Thus $F_{ij}\circ\mathrm{Dp}_{I}$ decomposes into a direct sum of projective functors.
Together with \eqref{eq5.2} and the statement of Lemma~\ref{lm6},
we get both claim~\eqref{lm17.1} and claim~\eqref{lm17.2}.
\end{proof}

\subsection{Duflo involutions}\label{s5.3}

We now consider the Duflo involutions corresponding to left cells in $ \mathcal{S}_{\mathscr{D}_A'} $.

\begin{theorem}\label{thm18}
{\tiny\hspace{2mm}}
\begin{enumerate}[$($i$)$]
\item\label{thm18.1} Let $I$ be an indecomposable ideal with $|G(I)|\neq1$. Then the element $\mathrm{Dp}_{J}$,
where $ J=\langle\varepsilon_i|\ i \in s_{G(I)}\rangle $,
is the Duflo involution corresponding to $\{[\mathrm{Dp}_{I}]\}$.
\item\label{thm18.2} For $j\in Q_0$, the element $F_{jj}$ is the Duflo involution for
$ \{[F_{ij}]|\ i\in Q_0\}$.
\end{enumerate}
\end{theorem}

\begin{proof}
Note that $\mathscr{D}_A$ is a $2$-subcategory of $\mathscr{D}_A'$. We claim that formula \eqref{eq3.1} gives
a projective presentation of $L_{\mathrm{Dp}_{J}}$ in $\overline{\mathbf{P}}_{\mathtt{i}}(\mathtt{i})$ for the
$2$-category $\mathscr{D}_A'$. Indeed, if there is no path from $j$ to $i$, all $2$-morphisms $P_{F_{ij}}\to P_{\mathrm{Dp}_{J}}$ are trivial since $\mathrm{Hom}_{A\text{-}A}(X_{ij},J)=0$. If there is a path from $j$ to $i$, the $A$-$A$-bimodule underlying $F_{ij}$ is an ideal in ${}_AA_A$, which is contained in the case considered in Proposition~\ref{pr9}. Therefore we obtain the claim and claim~\eqref{thm18.1} is proved similarly to Theorem~\ref{thm13}.

Claim~\eqref{thm18.2} is proved similarly to Theorem~\ref{thm14}. Let $I$ be an indecomposable ideal with
$|G(I)|\neq1$ and let
\begin{displaymath}
P\to P_{\mathrm{Dp}_{I}} \tto L_{\mathrm{Dp}_{I}}
\end{displaymath}
be a projective presentation of $L_{\mathrm{Dp}_{I}}$. Consider the bimodule map $T\to I$ underlying
$P\to P_{\mathrm{Dp}_{I}}$. There is  certainly a surjection from some projective
bimodule to $I$. Since $I$ itself is not projective, it follows that $T\to I$ is, in fact, a surjection.
Therefore the map $X_{jj}\otimes_A T\to X_{jj}\otimes_A I$ is surjective as well. But now all
direct summands of both $X_{jj}\otimes_A T$ and  $X_{jj}\otimes_A I$ are projective bimodules and
hence the latter map splits. This shows that $F_{ij}\,P\to F_{ij}\,P_{\mathrm{Dp}_{I}}$ is the zero object
in $\overline{\mathbf{P}}_{\mathtt{i}}(\mathtt{i})$ for $\mathscr{D}_A'$, meaning that
$F_{jj}\,L_{\mathrm{Dp}_{I}}=0$.

Similarly, following the proof of Theorem~\ref{thm14}, one shows that
$F_{jj}\,L_{F_{st}}=0$ for $s\neq j$ and that $F_{jj}\,L_{F_{jj}}\neq 0$.
Finally, exactly as at the end of the proof of Theorem~\ref{thm14}, one argues that the above implies
that $F_{jj}$ is the Duflo involution for $\{[F_{ij}]|\ i\in Q_0\}$. This completes the proof.
\end{proof}

\subsection{An alternative description}\label{s5.4}

Consider the {\em natural} left $A$-module $N$ defined by assigning the $1$-dimensional vector space $\Bbbk$
to each vertex of $Q$ and the identity linear transformation to each arrow in $Q$. Let
$N'$ be the {\em natural}  right $A$-module  defined by assigning the
$1$-dimensional vector space $\Bbbk$ to each vertex of $Q$ and the identity linear transformation to each
arrow in $Q^{\mathrm{op}}$. Consider the $A$-$A$-bimodule $N\otimes_{\Bbbk}N'$. Directly from the definition we have
\begin{equation}\label{eq5.3}
\dim \varepsilon_i  N\otimes_{\Bbbk}N'\varepsilon_j=1
\end{equation}
for all $i,j\in Q_0$. As $Q$ is a connected tree, the action graph of the $A\otimes_{\Bbbk}A^{\mathrm{op}}$-action
on $N\otimes_{\Bbbk}N'$ is connected, which, combined with \eqref{eq5.3}, implies that
the $A$-$A$-bimodule $N\otimes_{\Bbbk}N'$ is indecomposable. It turns out that
the $A$-$A$-bimodule $N\otimes_{\Bbbk}N'$ contains a lot of  information about
$1$-morphisms of the $2$-category $\mathscr{D}'_A$ as explained in the following proposition.

\begin{proposition}\label{pr19}
{\tiny\hspace{2mm}}

\begin{enumerate}[$($i$)$]
\item\label{pr19.1} For $i,j\in Q_0$, the unique (up to scalar) non-zero homomorphism
$\varphi_{ij}:X_{ij}\to N\otimes_{\Bbbk}N'$ is injective.
\item\label{pr19.2} The subbimodule $\sum_{i\in Q_0}\mathrm{Im}(\varphi_{ii})$ of $N\otimes_{\Bbbk}N'$
is isomorphic to ${}_AA_A$.
\end{enumerate}
\end{proposition}

\begin{proof}
By construction, each left projective $A$-module is a submodule of $N$ and each right projective
$A$-module is a submodule of $N'$. By tensoring over $\Bbbk$, which is exact, we thus get that each $X_{ij}$
is a submodule of $N\otimes_{\Bbbk}N'$. Therefore claim~\eqref{pr19.1} follows from
Lemma~\ref{lm7}\eqref{lm7.1}.

By construction, the subbimodule $\sum_{i\in Q_0}\mathrm{Im}(\varphi_{ii})$ has a basis given by
all paths in $A$. It is easy to check that, mapping such a path to the corresponding path in ${}_AA_A$
defines an isomorphism. This proves claim~\eqref{pr19.2}.
\end{proof}

After Proposition~\ref{pr19}, we may identify each $X_{ij}$ with the image of $\varphi_{ij}$
and also we may identify ${}_AA_A$ with the subbimodule $\sum_{i\in Q_0}\mathrm{Im}(\varphi_{ii})$
of $N\otimes_{\Bbbk}N'$. In this way all bimodules involved in the definition of $\mathscr{D}'_A$ are
realized as subbimodules  of $N\otimes_{\Bbbk}N'$.
For $i,j\in Q_0$ the intersection $X_{ij}\cap {}_AA_A$ is the unique subbimodule of both
$X_{ij}$ and ${}_AA_A$ which is maximal with respect to inclusions.
Denote by $\mathcal{T}$ the set $\{X_{ij}|\ i,j\in Q_0\}\cup\mathcal{I}(A)^{\mathrm{ind}}$.

\begin{corollary}\label{cr20}
Let $K,M\in\mathcal{T}$.

\begin{enumerate}[$($i$)$]
\item\label{cr20.1} We have $\dim\mathrm{Hom}_{A\text{-}A}(K,M)\leq 1$.
\item\label{cr20.2} Each non-zero homomorphism in $\mathrm{Hom}_{A\text{-}A}(K,M)$ is injective.
\item\label{cr20.3} We have $\mathrm{Hom}_{A\text{-}A}(K,M)\neq 0$ if and only if $K\subset M$
(inside $N\otimes_{\Bbbk}N'$).
\item\label{cr20.4} If $K\subset M$
(inside $N\otimes_{\Bbbk}N'$), then each non-zero homomorphism in $\mathrm{Hom}_{A\text{-}A}(K,M)$
is a scalar multiple of the natural inclusion.
\item\label{cr20.5} For any indecomposable $1$-morphisms $F,G\in \mathscr{D}'_A$  we have
$\dim\mathrm{Hom}(F,G)\leq 1$.
\end{enumerate}
\end{corollary}

\begin{proof}
Claims~\eqref{cr20.1}---\eqref{cr20.4} follows from Lemma~\ref{lm7}.
Claim~\eqref{cr20.5} follows from claim~\eqref{cr20.1}.
\end{proof}

By Proposition~\ref{pr19}, all composition multiplicities in  $N\otimes_{\Bbbk}N'$ are equal to one.
Since $\sum_{i,j\in Q_0}X_{ij}$ is a projective generator in the category of $A$-$A$-bimodules, it follows
that $\sum_{i,j\in Q_0}X_{ij}=N\otimes_{\Bbbk}N'$. Consequently, the set
$Z:=\{\varepsilon_i\otimes\varepsilon_j|\ i,j\in Q_0\}$
is a basis in $N\otimes_{\Bbbk}N'$ as, for $i,j\in Q_0$, the element $\varepsilon_i\otimes\varepsilon_j$
is a non-zero element in the one-dimensional space $\varepsilon_i  N\otimes_{\Bbbk}N'\varepsilon_j$.

Now we extend the notion of the partial order $\preceq$ on $Q^p$ (see Subsection~\ref{s2.2}) to $Z$.
For $w,w'\in Z$, we write $w\preceq w'$ if $w'=aw$ or $w'=wa$ for some $a\in Q^p$ and $w\prec w'$ if
$w\preceq w'$ and $w\neq w'$. For a subset $X\subset N\otimes_\Bbbk N'$ set $M_X=AXA$.
The following two statements are proved similarly to Lemmata~\ref{lm3} and \ref{lm4}, respectively.

\begin{lemma}\label{lm21}
{\tiny\hspace{2mm}}
\begin{enumerate}[$($i$)$]
\item\label{lm21.1} The map $X\mapsto M_X$ is a bijection from the set of anti-chains in $Z$ to the set of
all $A$-$A$-subbimodules in $N\otimes_\Bbbk N'$.
\item\label{lm21.2} For each anti-chain $X$ in $Z$, the set $X$ is a minimal set of generators in $M_X$ in the sense
that $M_Y\neq M_X$ for any proper subset $Y$ of $X$.
\end{enumerate}
\end{lemma}

\begin{lemma}\label{lm22}
Let $K\subset M$ be any two subbimodules of $N\otimes_\Bbbk N'$. Then $M$ covers $K$ if and only if $\dim(M/K)=1$.
\end{lemma}

The set $\mathcal{M}$ of all subbimodules in $N\otimes_\Bbbk N'$ is partially ordered with respect to inclusions.
Denote by $\mathcal{M}^\mathrm{ind}$ the subset of $\mathcal{M}$ consisting of all indecomposable subbimodules.
Then $\mathcal{M}^\mathrm{ind}$ inherits from $\mathcal{M}$ the poset structure. Note that
$\mathcal{T}\subset \mathcal{M}$. For $K, M\in\mathcal{T}$ we write $K\Subset M$ if $M$ covers $K$ in $\mathcal{T}$.
This is a proper analogue, in our situation, of the ind-covering relation in Subsection~\ref{s2.2}.
We complete this subsection with the following question:

\begin{question}\label{qn23}
The  $A$-$A$-bimodule $N\otimes_{\Bbbk}N'$ has finitely many subbimodules. Do these generate a
finitary $2$-category?
\end{question}

\subsection{Quiver for the underlying algebra of the principal $2$-representation of $\mathscr{D}_A'$}\label{s5.5}

Similarly to Subsection \ref{s3.3}, in this subsection we describe the quiver $\mathcal{Q}^{(2)}$ of
$\mathscr{D}_A'(\mathtt{i},\mathtt{i})$. To determine $\mathcal{Q}^{(2)}$, we use Corollary~\ref{cr20}.
The vertices of $\mathcal{Q}^{(2)}$ are elements in  $\mathcal{T}$ up to isomorphism. For any two indecomposable
subbimodules $K,M\in \mathcal{T}$, there is exactly one arrow from $M$ to $K$ if $K\Subset M$ and there are
no arrows otherwise. Furthermore, for any chains of subbimodules
\begin{displaymath}
M_1\Subset M_2 \Subset\dots \Subset M_k\quad\text{ and }\quad
N_1\Subset N_2 \Subset\dots \Subset N_m
\end{displaymath}
such that $M_1=N_1$ and $M_k=N_m$, we have
\begin{displaymath}
\iota_{(M_{k-1},M_k)}\cdots\iota_{(M_2,M_3)}\iota_{(M_1,M_2)} =
\iota_{(N_{m-1},N_m)}\cdots\iota_{(N_2,N_3)}\iota_{(N_1,N_2)}.
\end{displaymath}
Therefore we have to impose all commutativity relations on this quiver which make sense. Since all
$\iota_{(K,M)}$ are injective, no other relations are necessary.
Explicit examples are given in Subsection~\ref{s6.2}.

\section{Examples}\label{s6}

In this section we collect some explicit examples.
To illustrate subbimodules of the identity bimodule and $N\otimes_\Bbbk N'$, we follow the graphic convention of \cite[Section~6.5]{GM}, that is, the left action of arrows in $Q$ are depicted by solid arrows and the right action of arrows in $Q$ are depicted by dashed arrows.

\subsection{The algebra from Example~\ref{ex2}}\label{s6.1}

Consider the algebra from Example~\ref{ex2}.
The planar graph for identity bimodule $_AA_A$ is an orientation for
the signed Hasse diagram of $Q^p$, presented as follows:

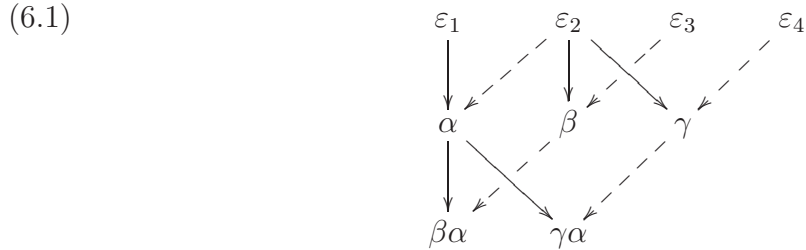
\begin{equation}\label{eq6.1}
\xymatrix{\varepsilon_1\ar[d] & \varepsilon_2 \ar[d]\ar@{-->}[dl]\ar[dr] & \varepsilon_3\ar@{-->}[dl]|\hole & \varepsilon_4\ar@{-->}[dl]\\
\alpha \ar[d]\ar[dr] & \beta \ar@{-->}[dl]|\hole & \gamma  \ar@{-->}[dl] &\\
\beta\alpha & \gamma\alpha  &&}
\end{equation}
Then indecomposable ideals of $A$ are connected full subgraph of \eqref{eq6.1} closed with respect to
the action of arrows (in the sense that if a subgraph contains some $w$ and there is an arrow, solid or dashed,
from $w$ to  $u$, then the subgraph contains $u$). For instance, ideals generated by one path of
zero length are listed here:
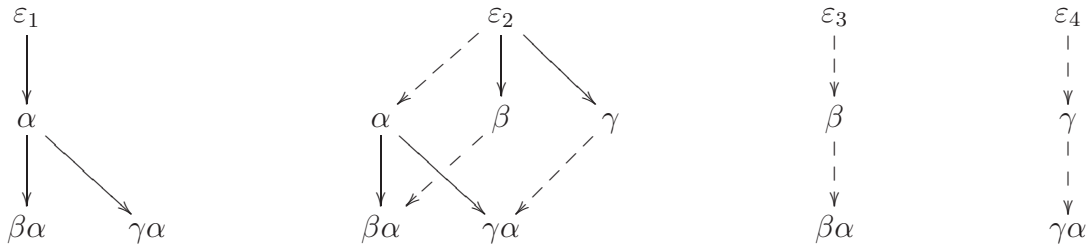
\begin{displaymath}
\xymatrix{\varepsilon_1\ar[d]&  &\quad&  &\varepsilon_2 \ar[d]\ar@{-->}[dl]\ar[dr] &   &\quad&  \varepsilon_3\ar@{-->}[d]& \quad& \varepsilon_4\ar@{-->}[d]\\
\alpha \ar[d]\ar[dr] &  &\quad&  \alpha\ar[d]\ar[dr]& \beta \ar@{-->}[dl]|\hole & \gamma\ar@{-->}[dl] &\quad& \beta\ar@{-->}[d]&\quad&
\gamma\ar@{-->}[d]\\
\beta\alpha &\gamma\alpha   &\quad&   \beta\alpha& \gamma\alpha &   &\quad&    \beta\alpha  &\quad& \gamma\alpha}
\end{displaymath}
and will be denoted by $I_1$, $I_2$, $I_3$, $I_4$, respectively.
For simplicity, for $a\in Q^p\setminus Q_0$ we denote by $I_a$ the ideal generated by $a$.
From Lemma~\ref{lm16}, each $I=\langle a \rangle$ for $a\in Q^p$ is isomorphic to $X_{t(a)s(a)}$.
For example, $I_\alpha\cong X_{21}$ and the corresponding $A$-$A$-bimodule isomorphism $\theta $ can be depicted in the following graph:
\begin{displaymath}
\xymatrix{ \beta\alpha&\overset{\theta}{\longrightarrow}&\beta\otimes \varepsilon_1 \\
\alpha\ar[u]\ar[d]&\overset{\theta}{\longrightarrow}&\varepsilon_2\otimes \varepsilon_1 \ar[u]\ar[d]\\
\gamma\alpha&\overset{\theta}{\longrightarrow}&\gamma\otimes \varepsilon_1}
\end{displaymath}

This following equation illustrates the proof of Lemma~\ref{lm17}. For $I=I_1+I_4+I_\beta$, we have
\begin{displaymath}
F_{i3}\circ\mathrm{Dp}_{I}\cong A\varepsilon_i\otimes_\Bbbk \varepsilon_3I\otimes_A{}_-
\cong  A\varepsilon_i\otimes_\Bbbk \beta A\otimes_A{}_-
\cong  F_{i2}.
\end{displaymath}
Note that $G(I)=\{\varepsilon_1,\varepsilon_4,\beta\}$ and the planar graph for $I$ is shown as follows:
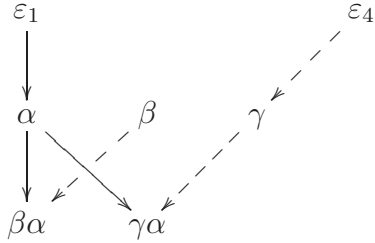
\begin{displaymath}
\xymatrix{\varepsilon_1\ar[d] &  && \varepsilon_4\ar@{-->}[dl]\\
\alpha \ar[d]\ar[dr] & \beta \ar@{-->}[dl]|\hole & \gamma  \ar@{-->}[dl] &\\
\beta\alpha & \gamma\alpha  &&}
\end{displaymath}

By Theorem~\ref{thm13}, the element  $\mathrm{Dp}_{J}$, where $J=\langle\varepsilon_1,\varepsilon_2,\varepsilon_4\rangle$, is the Duflo involution corresponding to the cell $\{[\mathrm{Dp}_I]\}$. The signed Hasse diagram for $J$ is:
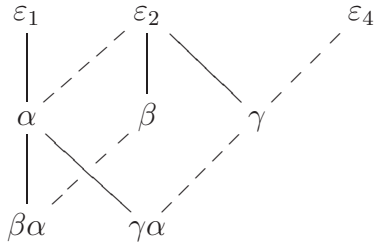
\begin{displaymath}
\xymatrix{\varepsilon_1\ar@{-}[d] & \varepsilon_2 \ar@{-}[d]\ar@{--}[dl]\ar@{-}[dr] && \varepsilon_4\ar@{--}[dl]\\
\alpha \ar@{-}[d]\ar@{-}[dr] & \beta \ar@{--}[dl]|\hole & \gamma  \ar@{--}[dl] &\\
\beta\alpha & \gamma\alpha  &&}
\end{displaymath}
From this diagram we have that the ideals covered by $J$ are $X_1:=\langle\varepsilon_2,\varepsilon_4\rangle $,
$X_2:=\langle\varepsilon_1,\varepsilon_4,\beta\rangle $ and $X_4:=\langle\varepsilon_1,\varepsilon_2\rangle $.
Since $IX_1=I_4\oplus I_\beta$, $IX_2=I_1+I_4$, $IX_4=I_1+I_\beta+I_\gamma$ and $IJ=I$,
then in this case all $\iota_{(IX_s,IJ)}$, where $s=1,2,4$, from \eqref{eq3.3} do not have direct summands
which are isomorphisms. This implies $\mathrm{Dp}_IL_{\mathrm{Dp}_{J}}\neq0$.

For any ideal $J'\not\subset J$, we know that $\mathrm{Dp}_IL_{\mathrm{Dp}_{J'}}=0$. Here is an illustration
for the latter. Consider $J'=\langle\varepsilon_1,\varepsilon_3\rangle\not\subset J$.
Then the ideals covered by $J'$ are $X'_1:=\langle\varepsilon_3,\alpha\rangle$ and
$X'_3:=\langle\varepsilon_1,\beta\rangle $.
We have $IX_1'=I_{\beta\alpha}\oplus I_{\gamma\alpha}$, $IX_3'=I_1$ and $IJ'=I_1$. Hence the map
$\iota_{(IX_3',IJ')}$ is the identity map and $\mathrm{Dp}_IL_{\mathrm{Dp}_{J'}}=0$.

\subsection{An example of an $A_2$-quiver}\label{s6.2}

Let $A=\Bbbk Q$, where $Q$ is given by:
$\xymatrix{1 \ar[r]^\alpha & 2&3\ar[l]_\beta}$.
The following graph represents the identity bimodule ${}_AA_A$:
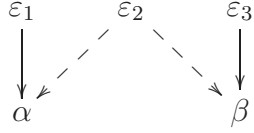
\begin{displaymath}
\xymatrix{\varepsilon_1\ar[d] & \varepsilon_2 \ar@{-->}[dl]\ar@{-->}[dr] & \varepsilon_3\ar[d] \\
\alpha & &\beta  }
\end{displaymath}
The graph for $N\otimes_\Bbbk N'$ is shown as:
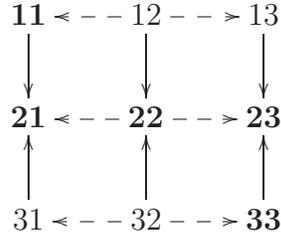
\begin{displaymath}
\xymatrix{\mathbf{11}\ar[d]&12\ar[d]\ar@{-->}[l]\ar@{-->}[r]&13\ar[d]\\
   \mathbf{21}&\mathbf{22}\ar@{-->}[l]\ar@{-->}[r]&\mathbf{23}\\
   31\ar[u]&32\ar@{-->}[l]\ar@{-->}[r]\ar[u]&\mathbf{33}\ar[u] }
\end{displaymath}
where $ij$ stands for $1$-dimensional $\Bbbk$-linear space with a basis $\varepsilon_i\otimes \varepsilon_j$.
The connected full subgraph with bold vertices describes the identity bimodule ${}_AA_A$.
The graph for the subbimodule $X_{ij}$ in $N\otimes_\Bbbk N'$ is the connected full subgraph whose vertices are
all vertices to which there is a path from $ij$. For simplicity, we abbreviate the subbimodule $X_{ij}$ by
$ij$. Then all indecomposable ideals (subbimodules) in $A$ are:
\begin{displaymath}
11,\ 22,\ 33,\ 21,\ 23,\
1122:=11+22,\ 2233:=22+33,\ 112233:=11+22+33.
\end{displaymath}

Note that $I_\alpha\cong 21, I_\beta\cong 23$ and $I_i\cong ii$ for $i=1,2,3$. Using the  chains
\begin{eqnarray}\label{eq6.2}
\begin{aligned}
21\Subset 11\Subset 1122\Subset 112233,&\quad&
23\Subset 22\Subset 2233\Subset 112233,\\
21\Subset 22\Subset 1122\Subset 112233,&\quad&
23\Subset 33\Subset 2233\Subset 112233,
\end{aligned}
\end{eqnarray}
we obtain that $\mathcal{Q}^{(1)}$ is the following graph:
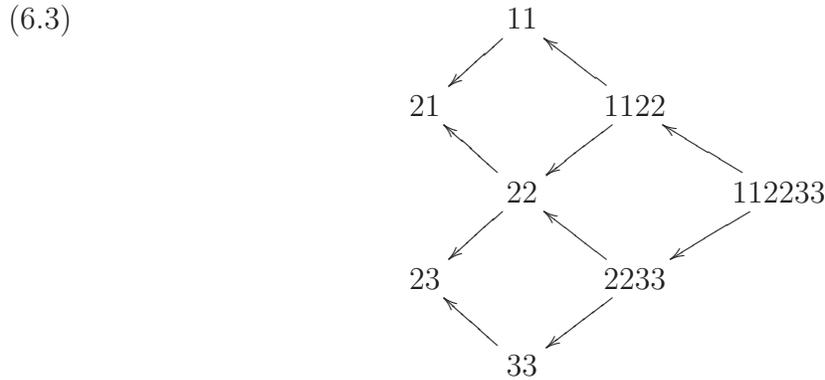
\begin{equation}\label{eq6.3}
\xymatrix@R=1.5pc@C=1.5pc{&11\ar[dl]&& \\
   21 &&1122\ar[ul]\ar[dl]& \\
   &22\ar[ul]\ar[dl]&&112233\ar[ul]\ar[dl]\\
   23 &&2233\ar[ul]\ar[dl]&  \\
   &33\ar[ul]&& }
\end{equation}
As mentioned in Subsection~\ref{s3.3}, for the quiver underlying the principal $2$-representation we have to
impose all commutativity relations in $\mathcal{Q}^{(1)}$.

All nonisomorphic indecomposable subbimodules in $\mathcal{T}$ are:
\begin{displaymath}
11,\ 22,\ 33,\ 21,\ 23,\ 1122,\ 2233,\ 112233,\ 12,\ 13,\ 31,\ 32.
\end{displaymath}
The quiver $\mathcal{Q}^{(2)}$ is then given by:
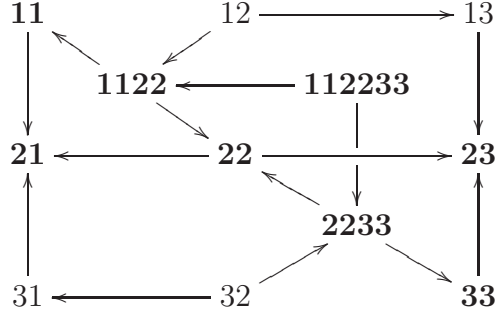
\begin{displaymath}
\xymatrix@R=1pc@C=1pc{\mathbf{11}\ar[dd]&&12\ar[dl]\ar[rr]&&13\ar[dd]\\
&\mathbf{1122}\ar[ul]\ar[dr]&&\mathbf{112233}\ar[dd]|\hole\ar[ll]&\\
   \mathbf{21}&&\mathbf{22}\ar[ll]\ar[rr]&&\mathbf{23}\\
   &&&\mathbf{2233}\ar[dr]\ar[ul]&\\
   31\ar[uu]&&32\ar[ll]\ar[ur]&&\mathbf{33}\ar[uu]}
\end{displaymath}
where we also impose all commutativity relations. It contains $\mathcal{Q}^{(1)}$ as a
subquiver given by the connected full subgraph with bold vertices. The additional to \eqref{eq6.2} covering relations are
\begin{displaymath}
23\Subset 13\Subset 12,\quad
1122\Subset 12,\quad
21\Subset 31 \Subset 32,\quad
2233\Subset 32.
\end{displaymath}

\subsection{Another example of an $A_2$ quiver}\label{s6.3}

\begin{example}\label{ex24}{\rm
If Q is given by: $\xymatrix{1 \ar[r]^\alpha &2\ar[r]^\beta&3 }$, then the corresponding identity module $_AA_A$ is  depicted as follows:
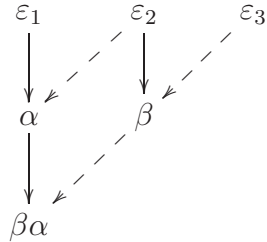
\begin{displaymath}
\xymatrix{\varepsilon_1\ar[d] & \varepsilon_2 \ar@{-->}[dl]\ar[d] & \varepsilon_3\ar@{-->}[dl] \\
\alpha\ar[d] &\beta\ar@{-->}[dl] &  \\
\beta\alpha&&}
\end{displaymath}
Following the same convention for notation as in Subsection~\ref{s6.2},
the graph for $N\otimes_\Bbbk N'$ is shown as:
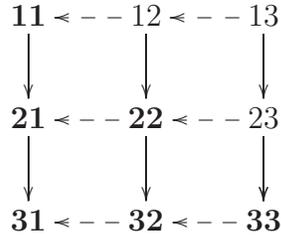
\begin{displaymath}
\xymatrix{\mathbf{11}\ar[d]&12\ar[d]\ar@{-->}[l]&13\ar[d]\ar@{-->}[l]\\
   \mathbf{21}\ar[d]&\mathbf{22}\ar[d]\ar@{-->}[l]&23\ar[d]\ar[d]\ar@{-->}[l]\\
   \mathbf{31}&\mathbf{32}\ar@{-->}[l]&\mathbf{33}\ar@{-->}[l] }
\end{displaymath}
The connected full subgraph with bold vertices gives the identity bimodule ${}_AA_A$.
All indecomposable ideals (subbimodules) in $A$ are:
\begin{eqnarray*}
&11,\ 22,\ 33,\ 21,\ 32,\ 31,\
1122,\ 1132:=11+32,\ 1133,\\
& 2233,\ 2132:=21+32,\ 2133:=21+33,\ 112233.
\end{eqnarray*}
Note that $I_\alpha\cong 21, I_\beta\cong 32, I_{\beta\alpha}\cong 31$ and $I_i\cong ii$ for $i=1,2,3$.
Using the chains
\begin{eqnarray*}
\begin{aligned}
31\Subset21\Subset11\Subset1132\Subset1122\Subset 112233,
&\quad&& 31\Subset32\Subset33\Subset2133\Subset 2233\Subset 112233,\\
31\Subset 21\Subset 2132\Subset 1132\Subset 1133\Subset 112233,
&\quad&&31\Subset 32\Subset 2132\Subset2133\Subset1133\Subset112233,\\
31\Subset21\Subset2132\Subset22\Subset1122\Subset112233,
&\quad&&31\Subset 32\Subset2132\Subset22\Subset2233\Subset112233,
\end{aligned}
\end{eqnarray*}
we compute all the arrows in $\mathcal{Q}^{(1)}$:
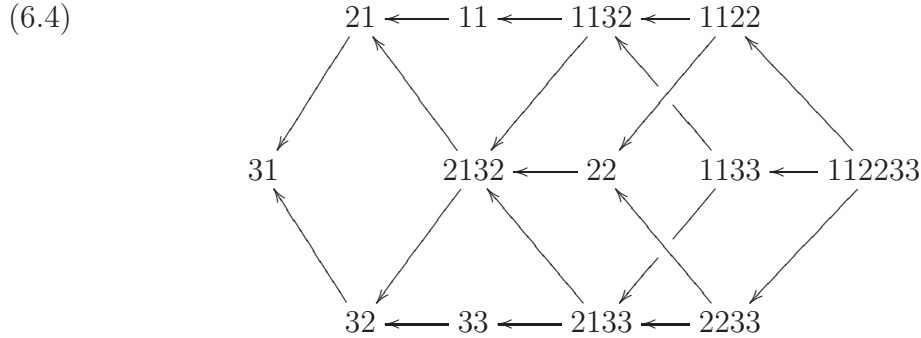
\begin{equation}\label{eq6.4}
\xymatrix@R=1.5pc@C=1.5pc{
&21\ar[ddl]&11\ar[l]&1132\ar[l]\ar[ddl]&1122\ar[l]\ar[ddl]&\\
&&&&&\\
31&&2132\ar[uul]\ar[ddl]&22\ar[l]&1133\ar[uul]|\hole\ar[ddl]|\hole&112233\ar[uul]\ar[ddl]\ar[l]\\
&&&&&&\\
&32\ar[uul]&33\ar[l]&2133\ar[l]\ar[uul]&2233\ar[l]\ar[uul]&\\
&&&&&&}
\end{equation}
and impose all commutativity relations. Note that $Q$ in this example is a subquiver of the quiver in
Example~\ref{ex2}, thus the graph for $\mathcal{Q}^{(1)}$ in Example~\ref{ex2} is larger than \eqref{eq6.4}.

All nonisomorphic indecomposable subbimodules in $\mathcal{T}$ are:
\begin{eqnarray*}
&11,\ 22,\ 33,\ 21,\ 32,\ 31,\
1122,\ 1132,\ 1133,\
2233,\
2132,\
2133,\ 112233,\ 12,\ 13,\ 23.
\end{eqnarray*}
The corresponding quiver $\mathcal{Q}^{(2)}$ is then given by:
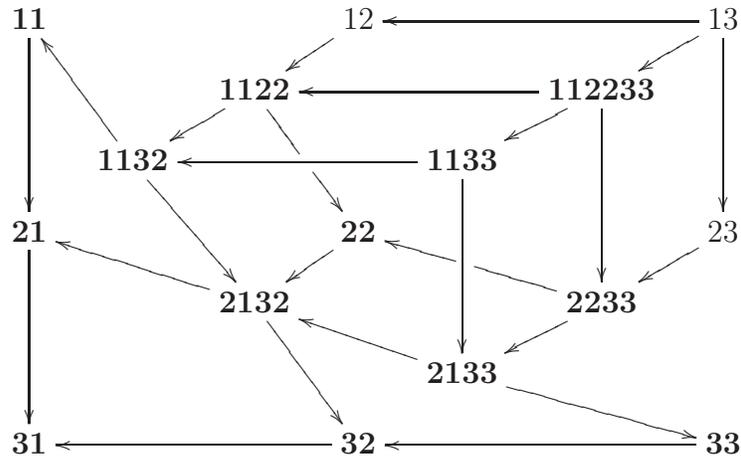
\begin{displaymath}
\xymatrix@R=1pc@C=1pc{
&\mathbf{11}\ar[ddd]&&&12\ar[dl]&&&13\ar[lll]\ar[ddd]\ar[dl]\\
&&&\mathbf{1122}\ar[dl]\ar[ddr]|\hole&&&\mathbf{112233}\ar[lll]\ar[dl]\ar[ddd]&\\
&&\mathbf{1132}\ar[uul]\ar[ddr]&&&\mathbf{1133}\ar[lll]\ar[ddd]&&\\
&\mathbf{21}\ar[ddd]&&&\mathbf{22}\ar[dl]&&&23\ar[dl]\\
&&&\mathbf{2132}\ar[ull]\ar[ddr]&&&\mathbf{2233}\ar[ull]|(.56)\hole\ar[dl]&\\
&&&&&\mathbf{2133}\ar[ull]\ar[drr]&&\\
&\mathbf{31}&&&\mathbf{32}\ar[lll]&&&\mathbf{33}\ar[lll]}
\end{displaymath}
with all possible commutativity relations. The connected full subgraph with bold vertices is exactly the quiver
\eqref{eq6.4}. Apart from  \eqref{eq6.4} we have the following covering relations in $\mathcal{T}$:
\begin{displaymath}
1122\Subset 12\Subset 13,\quad
112233\Subset 13,\quad
2233\Subset 23 \Subset 13.
\end{displaymath}
}
\end{example}

\noindent
Department of Mathematics, East China Normal University, Minhang District,
Dong Chuan Road 500, Shanghai, 200241, PR China,

E-mail address: {\tt scropure\symbol{64}126.com}.

\end{document}